\documentclass[11pt,amsfonts,amsmath]{article}
\usepackage[utf8]{inputenc}
\usepackage{times}
\usepackage{amssymb}
\usepackage{amsmath}
\usepackage{amsthm}
\usepackage{graphicx}
\usepackage{authblk}

\newcommand{\N}{{\mathbb{N}}}
\newcommand{\R}{{\mathbb{R}}}

\newcommand{\Q}{{\mathbb{Q}}}

\newtheorem{Theorem}{Theorem}
\newtheorem{Corollary}{Corollary}
\newtheorem{Definition}{Definition}
\newtheorem{Proposition}{Proposition}
\newtheorem{Remark}{Remark}

\newtheorem{lemma}{Lemma}
\newtheorem{example}{Example}
\def\eps{\varepsilon}
\begin{document}
\title{Heteroclinic cycles in Hopfield networks}
\author[1]{Pascal Chossat}
\author[2]{Maciej Krupa}
\affil[1]{\small Project Team  NeuroMathComp, INRIA-CNRS-UNS, 2004 route des lucioles, 06902 Valbonne, FRANCE} 
\affil[2]{\small Project Team MYCENAE,  INRIA-Rocquencourt, Domaine de Voluceau,
BP 105, 78153 Le Chesnay, France}

\maketitle 
\begin{abstract} 
\noindent Learning or memory formation are associated with the strengthening of the synaptic
connections between neurons according to a pattern reflected by the input. According to this
theory a retained memory sequence is associated to a dynamic pattern of the associated
neural circuit. In this work we consider a class of network neuron models, known as Hopfield
networks, with a learning rule which consists of transforming an information string
to a coupling pattern. Within this class of models we study dynamic patterns, known as robust heteroclinic
cycles, and establish a tight connection between their existence and the structure of the coupling.
\end{abstract}
{\bf keywords}: heteroclinic cycles, Hopfield networks, learning rule, network architecture. 
\section{Introduction} A simplest example of a heteroclinic cycle is a sequence of saddle type equilibrium points 
joined in a circle by connecting orbits. Generically heteroclinic cycles are not robust under perturbations of the system
(changes of parameters), but for special classes of systems they may occur robustly, typically due to the presence of invariant hyperplanes.
Examples of special structures leading to robust heteroclinic cycles are symmetry, the existence
of invariant planes corresponding to extinction of some species in Lotka-Volterra systems or existence of synchrony subspaces in coupled cell systems.
 Heteroclinic networks
are a generalization of heteroclinic cycles to sets of equilibria with more complicated connection structure.
More generally, heteroclinic sets may consist of invariant sets connecting periodic
or more complicated saddle type dynamics. The study of heteroclinic cycles
was motivated by examples in fluid mechanics (systems with symmetry) \cite{chossatlauterbach} and in population biology 
\cite{hofbauersigmund}.
See \cite{kruparev} for an introduction to the subject.

More recently, Rabinovich and co-workers have proposed applications of robust heteroclinic cycles in neuroscience, see
\cite{rabinovichetal} for an early review. Among the contexts proposed in \cite{rabinovichetal} 
where heteroclinic dynamics could be relevant
were central pattern generators (CPGs) and memory formation.
These two applications were validated by some more detailed biological studies \cite{rabinovichneuron} \cite{rabinovichPRL}.
The CPGs are circuits controlling the  motoric function, and are known to support a variety of complex oscillations
corresponding to different movements of the body. As shown in this paper, by slightly changing the coupling
structure in the model one can obtain a variety of heteroclinic cycles and thereby complex periodic solutions.
The idea of the memory application is similar -- modifications of the coupling, arising from the action of
the input, lead to the occurrence of periodic orbits, existing near heteroclinic cycles whose properties
reflect the structure of the input.

The focus of this work is to study Hopfield networks, which are the simplest models of memory circuits, with the goal of investigating the presence of heteroclinic cycles. \\
Hopfield introduced thirty years ago \cite{hopfield} this model for learning sequences and associative memory in neural networks, in their simplest possible form. In the continuous time version, for each neuron $i$ in a network of $N$ neurons, the activity is modeled by the following equation (activity model): 
 \begin{equation}\label{eq-hopnet}
 \dot u_i=-u_i - \sum_{i=1}^N J_{ij} g(u_j)+I_i,\qquad i=1,\ldots N.
 \end{equation}
 where $u_j$ is the activity variable (membrane potential) of neuron $j$ and $I_i$ is a constant external input on neuron $i$. 
The function $g :\R\to (0, 1)$ is strictly increasing and invertible, for example a sigmoid. The quantity $v_j=g( u_j)$ is the firing rate of neuron $j$, that is the time rate of spikes which are emitted by the neuron. Classically the function $g(u)=(1+e^{-u})^{-1}$ is used. The coupling coefficients $J_{ij}$ define a $N\times N$ matrix $J$ called the {\em connectivity matrix}. A positive (resp. negative) coefficient corresponds to an inhibitory (resp. excitatory) input from $j$ to $i$. When $J=0$ all neurons have the same state of rest $u_i=I_i$. When coupling is switched on, other equilibria may exist depending on the coefficients $J_{ij}$. It is often assumed that $J$ is a symmetric matrix, implying that the dynamics of the network always converges to an equilibrium. Each equilibrium is defined by a sequence of values $(u_1,\dots,u_N)$ called a {\em pattern}. Depending upon the inputs $I_i$, one or another equilibrium will be reached, a process that is interpreted as retrieving a pattern which has been earlier memorized through a tuning of the coupling coefficients in the network (Hebbian rule). \\
The assumption that $J$ is symmetric is unnecessary to storing information, and besides experiments have shown that dynamical patterns are often present in neural circuits and seem to play an important role in various aspects, in particular generating periodic oscillations in CPG. In a recent work Chuan et al. \cite{chuanetal} developed a method of converting strings of information, in the form of sequences of vectors with entries $\pm 1$, into coupling matrices so the Hopfield network with the resulting coupling architecture would have {\em storing cycles}, i.e. periodic orbits carrying the information of the underlying sequence of vectors. The defining feature of a storing cycle associated to such sequence is that it  visits the vicinity of each of its vector, preserving their order in the sequence. 
In another study Chuan et al. \cite{chuansiads} used Hopf bifurcation analysis to find storing cycles. \\
A natural observation is that heteroclinic cycles between equilibria given by the elements of the sequence provide a natural approximation to storing cycles. However the simulations of  \cite{chuanetal} and \cite{chuansiads} gave no evidence of the existence of heteroclinic cycles.

In this work we show that  after a small modification the systems studied in \cite{chuanetal} support heteroclinic
cycles. Our approach draws on the work of Fukai and Tanaka \cite{fukaitanaka}, 
who observed that by  of replacing a non-differentiable term in the firing rate equations by a constant one obtains a  Lotka-Volterra
system, which supports robust heteroclinic cycles \cite{hofbauersigmund}. This approach was subsequently used by  \cite{rabinovichetal} and \cite{ashwinetal} in their
study robust heteroclinic cycles in firing rate models.
In this work we continue the approach of  \cite{fukaitanaka}, introducing some refinements to their approximation of the firing rate equations.
We point out that the original firing rate equations cannot support hateroclinic cycles due to the presence of non-smooth terms and introduce
two methods of regularizing the equations. 
When the systems studied in \cite{chuanetal} are modified using either of our approaches
heteroclinic cycles do exist and there is a direct correspondence  between the input string/vector sequence/coupling structure
and the resulting heteroclinic cycle. In this work we carry out a detailed study of this correspondence.
 \section{Hopfield networks}
 \subsection{Storing cycles and network architecture}
 System \eqref{eq-hopnet} is often transformed to the {\em firing rate formulation}, by letting the firing rates $x_i=g(u_i)$
 be the dependent variables. In this section we make the same choice of $g$ as the authors of \cite{chuanetal}, namely
 \begin{equation}\label{eq-defgchuanetal}
 g(u)=\tanh(\lambda u),\quad\mbox{$\lambda$ is a parameter controlling the steepness of $g$.} 
 \end{equation} 
 In Section \ref{sec-LVapprox}, where we review some of the work of \cite{ashwinetal}, \cite{rabinovichPRL} and \cite{fukaitanaka}, we make a brief switch to a different but equivalent choice of $g$ used by these authors.
 System \eqref{eq-hopnet} transformed to the firing rate variables with $g$ given by \eqref{eq-defgchuanetal} has the form
  \begin{equation}\label{eq-stocyc}
 \dot x_j=(1-x_j^2)\left( \lambda(J{\mathbf x})_j - f(x_j)\right ), {\mathbf x}=(x_1,\ldots , x_n) \in [-1,\; 1]^n
 \end{equation}
 where
 \begin{equation}\label{eq-deff}
 f(x)=g^{-1}(x)={\rm arctanh}(x)=\frac12 \ln \left(\frac{1+x}{1-x}\right ) 
 \end{equation}
 and $J$ is the coupling matrix. We further decompose $J$ as follows:
 \begin{equation}\label{decompJ}
 J=c_0 I+c_1J_1
 \end{equation}
where $I$ is the identity matrix, $c_0$ and $c_1$ are non negative coefficients and $c_1=1-c_0$. \\
Provided that $\lambda c_0>1$ the equation $\lambda c_0\beta={\rm arctanh}(\beta)$ has a couple of non zero solutions $\pm\beta_\lambda$ with $0<\beta_\lambda<1$.
Therefore when $c_1=0$, any vector of the form $\beta_\lambda(\xi_1,\dots,\xi_n)$ with $\xi_j=\pm 1$ is a stable steady-state of  \eqref{eq-stocyc}. If we think of vectors of the form $(\xi_1,\dots,\xi_n)$ as information strings in a neural network, then the above steady-states represent stored memory states. However it is well-known that memory states need not be steady (see \cite{gencicetal} and references therein). If $c_1>0$ the steady-states may become unstable or even disappear, but nevertheless information may still be dynamically stored. 

 We now explain the idea of  information storage by means of limit cycles of \eqref{eq-stocyc} (storing cycles), as explored in \cite{chuanetal}
 and then we introduce our idea to use robust heteroclinic cycles instead. \\
 The basic question adressed in \cite{chuanetal} is the following: given an information string, can it be stored by a Hopfield network in the form
 of dynamic information, more specifically a limit cycle? Concretely, the information is given in the form of a string of binary $n$-vectors
 (with components equal to $\pm 1$). The learning rule, consistent with Hebbian learning, is an algorithm specifying how the information string structures the coupling matrix 
 $c_0I+c_1J$ (we forget from now on the subscript 1 in $J_1$).
 This learning rule will be described in detail in Section \ref{sec-eqnetarc}.
 The main research question of \cite{chuanetal} is whether the system with the coupling structure resulting
 from applying the learning rule supports stable limit cycles that code the original information string in the sense 
 that the periodic orbit passes through the quadrants of $\R^n$ corresponding to the elements of the information string, following its order.
 
 In this article we focus on a  different version of such
 encoding by the dynamics, choosing a robust heteroclinic cycle as the invariant object
 encoding the information string. The condition we impose is that the cycle
 should connect equilibrium points located at vertices of the cube $[-1,\; 1]^n$
 corresponding to the elements of the information string, following its order.
 This is a rather natural condition, yet the first obstacle we must overcome is that with $f$ as given by \eqref{eq-deff} the RHS of \eqref{eq-stocyc}
 is not $C^1$ on the cube $[-1,\; 1]^n$, so that heteroclinic cycles cannot exist.
 We discuss this problem in more detail and propose a solution in the next section,
 which also relates to the work of  \cite{ashwinetal}, \cite{rabinovichPRL} and \cite{fukaitanaka}.
\subsection{The Lokta-Volterra approximation to Hopfield equations}\label{sec-LVapprox}
The articles \cite{ashwinetal}, \cite{rabinovichPRL} and \cite{fukaitanaka} consider the question
of the existence of robust heteroclinic cycles in the firing rate version of \eqref{eq-hopnet} 
and show that such cycles exist for a Lotka-Volterra approximation of the system. In this section we 
use a different combination of $g$ and $f=g^{-1}$ consistent with choice made in these articles.
 Specifically we will use the functions: 
 \begin{equation}\label{eq-functions}
 g(u)=\frac{1}{1+e^{-u}}\quad\mbox{and}\quad f(x)=\ln \left(\frac{x}{1-x}\right ).
 \end{equation}
 The coefficients $J_{ij}$ are assumed to be all positive
 so that the synaptic couplings are all of inhibitory type.
 We define the firing rate by $x_j=g(\lambda^{-1} u_j)$ and transform \eqref{eq-hopnet}
 to the firing rate formulation.
 After applying a time rescaling we obtain the following system.
  \begin{equation}\label{eq-hopnetfr}
 \dot x_i=x_i(1-x_i)\left(-\lambda f(x_i) -\sum_{i=1}^n J_{ij} x_j+I\right ),\qquad i=1,\ldots n.
 \end{equation}
 Note that system \eqref{eq-hopnetfr} is well defined and continuous on the cube $[0, 1]^n$,
 but it is not smooth on the faces, with the term
 \begin{equation}\label{eq-ns}
x_i(1-x_i) f(x_i)
 \end{equation}
 being the source of non-smootheness.  As we are interested in heteroclinic cyles that lie on the edges, with connections
 in the faces, this becomes a problem for the existence and stability of the cycle.

 Since our purpose is merely to illustrate the problem of the lack of smoothness 
 we restrict our attention to the simplest case $n=3$. Then \eqref{eq-hopnetfr} has the form
  \begin{align}\label{eq-hopnetfr3}
  \begin{split}
 \dot x_1&=x_1(1-x_1)\left(-\lambda f(x_1) - J_{11} x_1-J_{12} x_2-J_{13} x_3+I\right )\\
 \dot x_2&=x_2(1-x_2)\left(-\lambda f(x_2) - J_{21} x_1-J_{22} x_2-J_{23} x_3+I\right )\\
 \dot x_3&=x_3(1-x_3)\left(-\lambda f(x_3)  - J_{31} x_1-J_{32} x_2-J_{33} x_3+I\right ).
 \end{split}
 \end{align}
 For simplicity we assume $J_{jj}=1$. The goal is to construct heteroclinic cycles connecting equilibria of the form:
 \[
 (\rho, 0, 0),\; (0, \rho, 0)\;\mbox{and}\; (0,0,\rho)\quad\mbox{where}\; -\lambda f(\rho)-\rho +I=0. 
 \]
The Jacobian matrix at such equilibria is given as follows:
\[
\left (\begin{array}{ccc} -\rho(1-\rho)-\lambda&-J_{12}(1-\rho)\rho &-J_{13}(1-\rho)\rho\\
                                                   0                & -\lambda f(0)-J_{21}\rho+I -\lambda&       0               \\
                                                   0                &               0                    & -\lambda f(0)-J_{31}\rho+I -\lambda\end{array}\right ).
\] 
Since $f$ blows up at $0$ the Jacobian is undefined. If the term \eqref{eq-ns} is neglected in the RHS of each
equation of \eqref{eq-hopnetfr3} then a heteroclinic cycle can be easily found, with
\begin{equation}\label{eq-Jeg}
J=\left (\begin{array}{ccc} 1 & 1.25 & 0\\
                                      0.875 & 1 & 1.25 \\
                                     3 & 0.625 & 1 \end{array}\right )
\end{equation}
giving an example \cite{ashwinetal}.
\subsection{Regularization}\label{sec-mod}
We propose two approaches to regularize the function $f$. First approach,
which we use in this paper, is to replace $f$ defined in \eqref{eq-functions} by its Taylor polynomial
at $x=0.5$. We denote such Taylor polynomial of degree $q$ as $f_q$.
Note that the sequence $\{f_q\}$ diverges at $0$ and $1$ when $q\to\infty$, but
converges uniformly to $f$ on any compact subinterval in $(0, 1)$. 

Another approach is to replace $f$ by
\begin{equation}\label{eq-feps}
f_\eps(x)=\ln \left(\frac{x+\eps}{1+\eps-x}\right ),
\end{equation}
where $\eps$ is a small parameter. The function $f_\eps$
is well defined on the interval $[0, 1]$, yet its  properties are similar to $f$,
in particular its derivative at $0$ and $1$ equals $1/(\eps(1+\eps)$, thus is very large.
If we replace $f$ by $f_q$ or $f_\eps$ in \eqref{eq-hopnetfr3} then, depending on the relative
size of $\lambda$ and $q$ or $\eps$, any of the three possibilities can arise:
\begin{enumerate}
\item the cycle does not exist,
\item the cycle exists and is unstable,
\item the cycle exists and is stable.
\end{enumerate}
Fig. \ref{fig} shows simulations of \eqref{eq-hopnetfr3} with $f$ replaced by $f_\eps$.
The matrix $J$ is as given in \eqref{eq-Jeg}, $I=0.8$ and $\lambda=0.01$. The value of $\eps$
is varied showing an example of each of the possible cases. 
\begin{figure}
\hspace{-1.2cm}\includegraphics[scale=0.6]{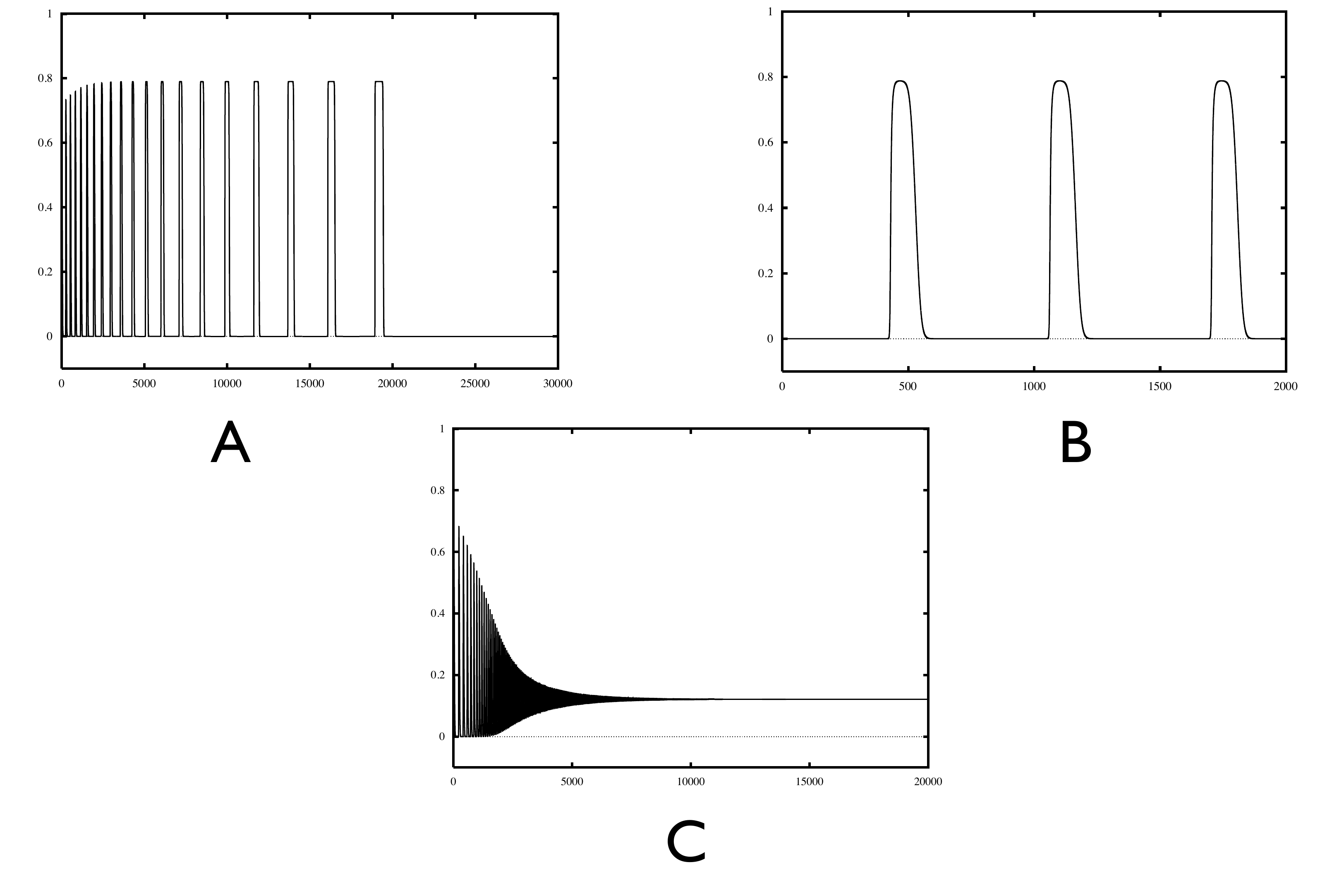}
\caption{Simulations with $J$ is as given in \eqref{eq-Jeg}, $I=0.8$ and $\lambda=0.01$
illustrate the three possible cases for the regularized system corresponding to three different values of $\eps$. 
Panel {\bf A} shows a heteroclinic
cycle, with $\eps=0.12$, panel {\bf B} shows a periodic orbit close to an unstable cycle, with $\eps=0.07$, panel {\bf C} shows the dynamics
attracted to an equilibrium for the case when no cycle exists, with $\eps=0.02$.}
\label{fig}
\end{figure}

\section{Hopfield networks with coupling given by the learning rule of \cite{personnazetal}}\label{sec-storcyc} 
\subsection{The equations and network architecture}\label{sec-eqnetarc}
We now return to the formulation \eqref{eq-stocyc}  with $f$ replaced by its $q$-th order polynomial expansion $f_q$ at $x=0$ as described in Section \ref{sec-mod}. The equation now reads
\begin{equation}\label{eq-f_q}
 \dot {\mathbf x}=(I-\mbox{ diag}\; {\mathbf x}\cdot {\mathbf x})\left(\lambda (c_0 {\mathbf x}+c_1J{\mathbf x}) - f_q({\mathbf x})\right ), {\mathbf x}\in [-1,1]^n,
\end{equation} 
with $f_q( {\mathbf x})=(f_q(x_1),\ldots,f_q(x_n))$ and 
$$f_q(x) = x+\frac{x^3}{3}+\cdots+\frac{x^q}{q}.$$
The power series $f_\infty$ has a radius of convergence equal to 1. It follows that given any interval $(-1+\epsilon,1-\epsilon)$, the approximation of $f$ by $f_q$ can be as good as we wish provided that $q$ is large enough. 

We now give a formal description of the information string and introduce the learning rule. \\
 A {\em binary pattern} (or simply a {\em pattern}) is a vector $\xi$ of binary states of $n$ neurons: $\xi=(x_1,\ldots,x_n)^t$ with $x_j=\pm 1$.
Let 
 \begin{equation}\label{eq:transitionmatrix}
 \Sigma=\left(\xi^1,\xi^2,\ldots ,\xi^p\right)
 \end{equation} 
 be a sequence of $p$ patterns. This matrix is called a {\em cycle} if there exists a connectivity matrix $J$ such that the corresponding network of $n$ neurons visits sequentially and cyclically the patterns defined by $\Sigma$.
In other words each column $\xi^j$ can be associated with a state of the system such that the signs of the cell variables are equal
to the signs of the corresponding components of $\xi^j$. We shall always assume $p\geq n$. \\
Let $P$ be the matrix of the cyclic permutation $(x_1, x_2, \ldots, x_{p-1}, x_p)\to (x_2, x_3, \ldots, x_p, x_1)$.
The cycle $\Sigma$ is called {\em admissible} if there exists $J$ such that
 \begin{equation}\label{eq-defsig}
 J\Sigma=\Sigma P,
 \end{equation}
has a solution \cite{personnazetal}. This relation expresses a necessary condition for the network \eqref{eq-stocyc} to possess a solution that periodically takes the signs defined by the patterns $\xi_1,\ldots,\xi_p$. Note that, if $\Sigma$ is admissible then a solution exists in the form
 \begin{equation}\label{sol-pseudoinverse}
 J=\Sigma P\Sigma^+
 \end{equation}
where $\Sigma^+$ is the Moore-Penrose pseudo-inverse of $\Sigma$, and if $\Sigma$ has full rank  it is unique. 

A cycle $\Sigma$ is called {\em simple} \cite{chuanetal} if there exists a vector $\eta\in\R^p$ such that each row of $\Sigma$ equals $\eta P^{l_j}$, for some $p>l_j\ge 0$. 
We define 
\begin{equation}\label{eq-defW}
W_\eta={\rm span}\; \{\eta P^j\, :\, j=0,\ldots, p-1\}.
\end{equation}
By Theorem 2 in \cite{chuanetal} a simple cycle is admissible if and only if ${\rm dim}\; W_\eta=Rank(\Sigma)$. \\
If in addition we can write 
 \begin{equation}\label{eq:consecutiveSigma}
 \Sigma=\left(\eta^t,(\eta P)^t,\ldots ,(\eta P^{n-1})^t\right)^t
 \end{equation}
then the simple cycle is called {\em consecutive}. The following proposition is essentially contained in Sec. 5.1.1 of \cite{chuanetal}.
\begin{Proposition} \label{prop:Jconsecutivecycle}
If a simple consecutive cycle is admissible, then there exists a $J$ satisfying \eqref{eq-defsig} of the form
\begin{equation}\label{eq:Js}
 J=\left (\begin{array}{rrrrrr} 0&1&0&\ldots&0&0\\0&0&1&\ldots&0&0\\\vdots&\vdots&\vdots&\ddots&\ddots\\0&0&0&\ldots&0\\a_0&a_1&a_1&\ldots&a_{n-2}&a_{n-1}\end{array}\right ).
 \end{equation}
 where $a_0,\ldots,a_{n-1}$ are rational coefficients. If $\Sigma$ has full rank then $J$ is uniquely defined and $a_0\neq 0$ (in this case the cycle is minimal in the sense of \cite{chuanetal}).
\end{Proposition}   
\begin{proof}
By construction we can write 
$$\Sigma = \left (\begin{array}{l}\eta \\ \eta P \\ \vdots \\ \eta P^{n-1}\end{array}\right ) \text{~~and~~~~} \Sigma P = \left (\begin{array}{l}\eta P \\ \eta P^2 \\ \vdots \\ \eta P^n \end{array}\right ).$$
By admissibility $\Sigma P=J\Sigma$ and moreover $\eta P^n$ must be a linear combination of the $\eta P^j$'s. Hence \eqref{eq:Js} follows. The $a_j$'s are rational because the vectors $\eta P^j$ have integer coordinates. 
If $Rank(\Sigma)=n$ then $J$ is non singular, hence $a_0\neq 0$.
\end{proof}
\begin{example}\label{ex-basic}
We consider $\Sigma$ as follows, with $p=6$:
 \begin{equation}\label{eq-seq6}
 \Sigma=\left (\begin{array}{rrrrrr}1&1&1&-1&-1&-1\\1&1&-1&-1&-1&1\\1&-1&-1&-1&1&1\end{array}\right ),
  \end{equation}
  Let $\eta=(1,1,1,-1,-1,-1)$. Note that the rows of $\Sigma$ are $\eta$, $\eta P$ and $\eta P^2$. Note also that $\eta P^3=-\eta$. It follows that the rows of $\Sigma P$ are 
  $\eta P$, $\eta P^2$ and $-\eta$, i.e. the second, the third and the negative of the first row of $\Sigma$. Hence
  \begin{equation}\label{eq-temp}
  \left (\begin{array}{rrr} 0&1&0\\0&0&1\\-1&0&0\end{array}\right ) \Sigma =\Sigma P.
  \end{equation}
  Since the rows of $\Sigma$ are independent the matrix $\Sigma\Sigma^T$ is invertible. Hence
  \eqref{eq-defsig} has a unique solution which, by \eqref{eq-temp}, must be given by:
  \begin{equation}\label{eq:Jn=3}
 J=\left (\begin{array}{rrr} 0&1&0\\0&0&1\\-1&0&0\end{array}\right ).
 \end{equation}
 Note that  $\Sigma^+=\Sigma^T(\Sigma\Sigma^T)^{-1}$ and $J$ satisfies \eqref{sol-pseudoinverse}.
 
 This matrix provides a simple example of heteroclinic cycle, which we illustrate in Fig. \ref{hetcyc36}: the reader can check on this numerical simulation that indeed trajectories follow the pattern defined by $\Sigma$. Observe that the trajectory closely follows the edges of the cube connecting the equilibria in the pattern. The analysis is easy but it follows directly from Proposition \eqref{prop:condition edge cycle} in Section 4.2 (see Example \ref{ex:edgecycle}).
\begin{figure}[h]
\center\includegraphics[scale=0.45]{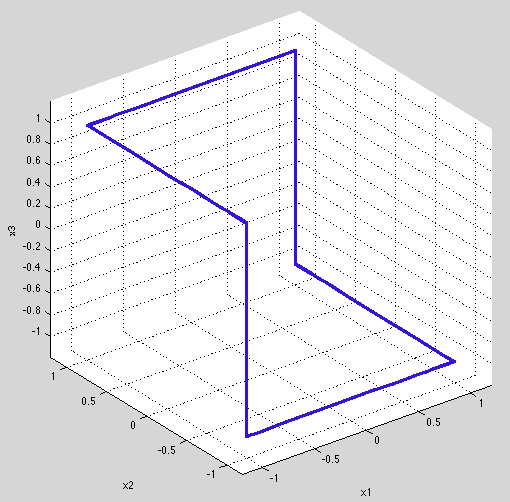}
\caption{A trajectory of \eqref{eq-stocyc} with $J$ given by \eqref{eq:Jn=3} and with $c_0=0.6$, $\lambda=8$. Initial conditions close to $(1,1,1)$.}
\label{hetcyc36}
\end{figure}
  
 \end{example}
\subsection{Classification of simple consecutive cycles}\label{sec-clahop}
Suppose that $p$ is fixed and note that every consecutive cycle is uniquely determined by the choice 
of $\eta$ and $n$. If $n\ge p$ then such a cycle is always admissible.
If $n<p$ then there are only very special choices of $\eta$ and $n$ such that the cycle is admissible. In this section we will address the question
of finding the conditions on $\eta$ so that there exists an $n<p$ such that the cycle determined by $\eta$ and $n$ is admissible. 
In order to avoid confusion with prime numbers we will, throughout this section, use the letter $m$ instead of
$p$ to denote the dimension of $\eta$. We will return to the original notation of \cite{chuanetal} in the subsequent sections. 

\noindent Consider a simple cycle as defined in Section \ref{sec-eqnetarc}, with $\eta$ corresponding to the first row of $\Sigma$.
If $\Sigma$ is admissible the there exists $J$ in the form given by \eqref{eq:Js} such that  \eqref{eq-defsig} is satisfied. Let $(a_0, a_1, \ldots, a_{n-1})$ be the
last row of $J$ (see \eqref{eq:Js}) and let 
\begin{equation} \label{eq-defpolypsi}
\psi(x)=x^n-a_{n-1}x^{n-1}-\ldots -a_1x-a_0.
\end{equation}
It follows from  \eqref{eq-defsig} that $\eta\in\ker(\psi(P))$.
In this section we use the following result:
\begin{Theorem}\label{th-charinvsp}
Let $V$ be a non-trivial invariant subspace of the action of $P$ on $\R^m$. Then there exists a polynomial $\phi(x)$, which is a divisor of $x^m-1$,
such that $V=\ker(\phi(P))$. Moreover, for any $\tilde \phi$ the inclusion $V\subset \ker(\tilde\phi(P))$ holds if and only if $\phi$ is a divisor of $\tilde\phi$.
\end{Theorem}
\noindent Theorem \ref{th-charinvsp} follows from some classical results of algebra, which we will review in the appendix, thereby providing the proof.
We now state two corollaries of Theorem \ref{th-charinvsp} which we will use to characterize the possible choices of $\eta$ for which ${\rm dim} (W_\eta) < m$.
\begin{Corollary}\label{cor-weta1}
If $n={\rm dim}\; (W_\eta)<m$ then $W_\eta=\ker(\psi(P))$ and $\psi$ is a divisor of $x^m-1$.
\end{Corollary}
\noindent {\bf Proof}. It is easy to see that $W_\eta\subset \ker(\psi(P))$. We will prove that the opposite inclusion holds 
and that $\psi$ is a divisor of $x^m-1$. By Theorem \ref{th-charinvsp} there exists $\phi$ a divisor of $x^m-1$ 
such that $\ker(\phi(P))=W_\eta$ and $\phi$ divides $\psi$. Suppose that $\phi$ is a proper divisor of $\psi$.
Then 
\[
n={\rm dim}\; (W_\eta)\le \deg(\phi)<n,
\]
which is a contradiction. It follows that $\psi=\phi$. Hence the corollary holds.\\[1ex]
\noindent  For $\phi$ a minimal divisor of $x^m-1$ (over ${\mathbb Q}$) let $\psi=(x^m-1)/\phi$ and let $W_\phi= \ker(\psi)$.
The following result  leads to a characterization of $\eta$s such that $n={\rm dim} (W_\eta)<m$.
\begin{Corollary}\label{cor-weta2}
If $n={\rm dim} (W_\eta)<m$ then $\eta\in W_\phi$, for some $\phi$ a minimal divisor of $x^m-1$ (over ${\mathbb Q}$).
\end{Corollary}
\noindent {\bf Proof} By Corollary \ref{cor-weta1} there exists $\tilde \psi$, a divisor of $x^m-1$, such that $W_\eta={\rm ker}(\tilde\psi)$.
By unique decomposition into prime factors over $\Q$ there exists a minimal divisor $\phi$ of $x^m-1$ which divides $\tilde\psi$.
Let $\psi=(x^m-1)/\phi$. Clearly $\tilde \psi$ divides $\psi$. It follows from Theorem \ref{th-charinvsp} that $\eta\in{\rm ker}(\psi)=W_\phi$.\\[1ex]
\noindent
In the remainder of this section we will derive the conditions  on $\eta$ needed for $\eta\in W_\phi$ for some (the simplest)
choices of $\phi$, where $\phi$ is  a minimal divisor of $x^n-1$.
We begin by recalling the decomposition of $x^n-1$ into irreducible polynomials over $\mathbb{Q}$. For a positive integer $k$ let
\[
\phi_k(x)=\sum_{i=0}^{k-1}x^i.
\]
The polynomials $\phi_p$, where $p$ is a prime number, are irreducible over $\mathbb{Q}$.
For a prime number $p$ and a non-negative integer $j$ we define
\[
\phi_{p,j}(x)=\sum_{i=0}^{p-1}x^{ip^j}.
\]
Note that $\phi_{p,1}=\phi_p$. The polynomials $\phi_{p,j}(x)$ are the irreducible factors over $\mathbb{Q}$ of the polynomial $\phi_{p^k}$.
Suppose $m=m_1m_2\ldots m_l$, with $m_j=p_j^{k_j}$, $j=1,\ldots, l$, and let
\[
\Phi_m(x)=\phi_m\big /\prod_{j=1}^l \phi_{m_j}(x).
\]
The polynomial $\Phi_m(x)$ is called the cyclotomic polynomial of degree $m$ and is irreducible. It now follows that the decomposition of $x^m-1$
into irreducible factors over  $\mathbb{Q}$ given by
\begin{equation}\label{eq-facto}
x^m-1=(x-1)\Phi_m(x)\prod_{j=1}^l\left (\prod_{i=0}^{k_j-1} \phi_{p_j,i}(x)\right ). 
\end{equation}
All the possible factors of $x^m-1$ over $\mathbb{Q}$ are products of the irreducible factors appearing in \eqref{eq-facto},
hence all the possible choices of $W_\phi$ are obtained that way. 
As announced above we now describe some of the spaces $W_\phi$ by simple conditions on the components of $\eta$.
\begin{Proposition}\label{prop-cond1}
If $\phi=\phi_{p_j}$, for some prime number $p_j$ then 
$W_\phi$ consists of vectors $\eta=(b_0,\ldots , b_{m-1})$ satisfying 
\begin{equation}\label{eq-cond1}
\sum_{i=0}^{m/p_l-1} b_{ip_j}=\sum_{i=0}^{m/p_l-1} b_{ip_j+1}=\ldots \sum_{i=0}^{m/p_l-1} b_{ip_j+p_j-1}.
\end{equation}
\end{Proposition}
\noindent {\bf Proof}
 we use the following identity:
\[
\psi(x)=(x-1)\phi_{m/p_j}(x^{p_j}).
\]
Hence
\begin{equation}\label{eq-tocond1}
\Psi(P)\eta=(P-I)\left ( \sum_{i=0}^{m/p_l-1} b_{ip_l}, \sum_{i=0}^{m/p_l-1} b_{ip_l+1},\ldots , \sum_{i=0}^{m/p_l-1} b_{ip_l+m-1}\right ).
\end{equation}
(The indexing of the components of $b$ in \eqref{eq-tocond1} must be understood modulo $m$.)
It follows that the RHS of \eqref{eq-tocond1} is equal to the $0$ vector if \eqref{eq-cond1} holds. \\[1ex]
\noindent We now state the condition on $\eta$ for $\phi= \Phi_m(x)$. 
 We begin with the following elementary lemma (the proof is left to the reader).
\begin{lemma}\label{lem-elem}
 If $k$ divides $m$ then
\begin{equation}\label{eq-char}
\ker(P^k-I)=\{\eta\; :\; \mbox{ there exists $v\in \{-1,1\}^k$ such that }\; \eta=(v,v,\ldots, v)\}.
\end{equation}
\end{lemma}
\begin{Proposition}\label{prop-cond2}
Suppose $\phi= \Phi_m(x)$. Then 
\begin{equation}\label{eq-cond2}
\eta\in\ker(P^{m_1}-I)+\ldots+\ker(P^{m_l}-I).
\end{equation}
\end{Proposition}
\noindent {\bf Proof} Note that
\[
\psi(x)=(x^m-1)\big / \Phi_m(x)=(x-1)\prod_{j=1}^l \phi_{m_j}(x).
\]
Further note that, for each $j$, $(x-1)\phi_{m_j}(x)=x^{m_j}-1$. Hence, for each $j$,
\[
\ker(P^{m_j}-I)\subset \ker(\psi(P)).
\]
Moreover, for $j\neq j'$
\[
\ker(P^{m_j}-I)\cap\ker(P^{m_{j'}}-I)=\rm span\{\mathbf{1}_m\}.
\]
The result follows.
\begin{Remark}\label{rem-oneblockspace}
Since the coordinates of 
$\eta$ are $\pm 1$, it follows that $\eta$ must be contained in one of the spaces $\ker(P^{m_l}-I)$.
\end{Remark}
\begin{Remark}\label{rem-other}
 The conditions for the other minimal factors of $x^m-1$ are slightly more complicated and we will not state them here. They
are, however, not hard to derive.
\end{Remark}
\begin{example}\label{ex-prop2}
We consider $\eta=(b_0,b_1,\ldots, b_{14})=(1,1,1,1,1,1,1-1,-1,-1,1,1,1,1,1)$ with $m=15=3\cdot 5$. Note that
$\eta$ satisfies \eqref{eq-cond1} with $p=3$. Hence, by Proposition \ref{prop-cond1}, $\eta\in W_\phi={\rm ker}(\psi)$ with 
\[
\phi=1+x+x^2;\qquad \psi=(x-1)(1+x^3+x^6+x^9+x^{12})=-1+x-x^3+x^4-x^6+x^7-x^9+x^{10}-x^{12}+x^{13}
\]
Hence 
\begin{equation}\label{eq-SPJS}
\eta P^{13}= \eta -\eta P+\eta P^3-\eta P^4+\eta P^6-\eta P^7 +\eta P^9-\eta P^{10} +\eta P^{12}.
\end{equation}
Note that the last row of $\Sigma P$ is equal to the LHS of \eqref{eq-SPJS}. If we define $J$ as in \eqref{eq:Js}
with $(a_0,\ldots, a_{12})=(1,-1,0,1,-1,0,1,-1,0,1,-1,0,1)$ then the RHS of \eqref{eq-SPJS} equals $J\Sigma$.
Hence, by \eqref{eq-SPJS}, identity \eqref{eq-defsig} holds with $J$ as specified.
\end{example}
\begin{example}\label{ex-prop3}
We consider $\eta=(b_0,b_1,\ldots, b_{14})=(1,1,1,-1,-1,1,1,1,-1,-1,1,1,1,-1,-1)$ with $m=15=3\cdot 5$. Note that
$\eta\in ker(P^5-I)$, or, in other words,
\begin{equation}\label{eq-SPJSprime}
\eta P^{5}= \eta.
\end{equation}
Arguing as in Example \ref{ex-prop3} we conclude that $\Sigma$ generated by $\eta$ and $n=5$ is admissible
with $J$ whose last row equals $(1,0,0,0,0)$.
\end{example}
\begin{example}\label{ex-n/2}
An interesting class of admissible cycles exists for $m$ even with $n=m/2$. Note that in this case $x^m-1=(x^n+1)(x^n-1)$, i.e. $x^n+1$ divides $x^m-1$. Further note that
\begin{equation}\label{eq-n/2}
{\rm ker}\; (P^n+I)=\{\eta\in\R^m=(\nu,-\nu),\quad \nu\in\R^n\}.
\end{equation}
Let $\Sigma$ be a cycle constructed with some $\eta\in {\rm ker}\; (P^n+I)$, $n=m/2$. Then $\Sigma$ is admissible.
Moreover, by a similar argument as in Example \ref{ex-prop2} we conclude that the last row of
$J$ equals $(-1,0,\ldots, 0)$. In particular Example  \ref{ex-basic} of Section \ref{sec-eqnetarc} is a special case of this construction. This type of admissible cycle is called {\em antisymmetric} in \cite{chuanetal}.
\end{example}
\begin{example}\label{ex-xm1}
Since $(1+x+\ldots +x^{m-1})(x-1)=x^m-1$ the space ${\rm ker} (I+P\ldots +P^{m-1})$ corresponds to vectors $\eta$ for which
$\Sigma$ with $n=m-1$ is admissible. Moreover
\[
{\rm ker} (I+P\ldots +P^{m-1})=\{\eta\; :\; \sum\eta_i =0\}.
\]
Note that for $\eta$'s whose entries are $\pm 1$ this means that the number of coordinates equal to $1$ is the same
as the number of coordinates equal to $-1$. Hence $m$ must be even. In this case the last row of $J$
is $(-1,-1,\ldots, -1)$.
\end{example}
%
 
\section{Consecutive Hopfield cycles and their heteroclinic cycles}
We now come to the study of heteroclinic cycles for admissible consecutive simple cycles governed by equation (\ref{eq-stocyc}), hence with $J$ as in (\ref{eq:Js}). 
Then the equation reads as a system
  \begin{align}\label{eq-consecutive}
  \begin{split}
 \dot x_1&=(1-x_1^2)\left(\lambda c_0x_1 + \lambda c_1x_2 - f_q(x_1)\right )\\
 \dot x_2&=(1-x_2^2)\left(\lambda c_0x_2 + \lambda c_1x_3 - f_q(x_2)\right )\\
 \vdots &  \\
 \dot x_n &=(1-x_n^2)\left(\lambda c_0x_n + \lambda c_1(a_1x_1+\dots+a_nx_n)- f_q(x_n)\right )
 \end{split}
 \end{align}
Following \cite{chuanetal}, we also assume that the two coefficients which control the relative contributions of $J_0$ and $J$ to each neuron satisfy 
\medskip

{\bf (H)}  $0\leq c_0< 1$ and $c_0+c_1=1$. 
\medskip

We aim at studying the existence and stability of heteroclinic cycles connecting vertex equilibria, i.e. equilibria with entries $\pm 1$, for this system. By construction, the edges, faces and simplices of the hypercube $\{(\pm 1,\dots,\pm 1)\}$ are invariant under the dynamics of \ref{eq-consecutive}. \\
Let $\xi=(x_1,\dots,x_n)$ be a {\em vertex equilibrium}: $x_k=\pm 1$ for all $k$. Linearizing (\ref{eq-consecutive}) at $\xi$ leads to a system of equations $\dot u_k=\sigma_k u_k$ where: we can express the eigenvalues as follows:
 \begin{equation}\label{eigenvalues}
 \begin{array}{ccl}
 \sigma_k &=& 2(f_q(1)-\lambda)~~~~\text{if}~x_kx_{k+1}=1,~k<n \\ \sigma_k &=& 2(f_q(1)-\lambda(c_0-c_1))~~~~\text{if}~x_kx_{k+1}=-1,~k<n \\ \sigma_n &=& 2\left(f_q(1)-\lambda(c_0 + c_1x_n\sum_{j=1}^n a_jx_j)\right ),
 \end{array}
 \end{equation}
Note that under the above conditions on $c_0$ and $c_1$, which we assume from now on, a necessary and sufficient condition for the existence of negative {\em and} positive eigenvalues with $k<n$ is that 
\begin{equation}\label{eq:cond_signes}
\lambda(c_0-c_1)<f_q(1)<\lambda
\end{equation} 
This is always possible to realize since $|c_0-c_1|<1$. Then $\sigma_k<0$ if $x_kx_{k+1}=1$ and $>0$ otherwise.

\begin{Remark}
The equation \eqref{eq-f_q} (hence \eqref{eq-consecutive}) is invariant by the symmetry $S:~{\bf x}\rightarrow -{\bf x}$. Therefore any time a cycle $\Sigma$ admits a heteroclinic cycle, the cycle $-\Sigma$ admits the {\em opposite} heteroclinic cycle obtained by applying $S$. \\
In the following we shall always consider $\Sigma$'s up to this symmetry. 
\end{Remark}

\subsection{Heteroclinic edge cycles}
 \begin{Definition} \label{def:edgecycle}
 A heteroclinic cycle is called an "edge cycle" if it connects a cyclic sequence of vertex equilibria through heteroclinic orbits lying on the edges of the hypercube $\{(\pm 1,\dots,\pm 1)\}$. We also request that the unstable manifold at each equilibrium in the cycle has dimension 1 (therefore is contained in an edge).
 \end{Definition}
The condition about the unstable manifolds is necessary for asymptotic stability of the edge cycles. If $\sigma_k^-$ and $\sigma_k^+$ denote respectively the contracting and expanding eigenvalues along the heteroclinic trajectories, the edge cycle is asymptotically stable if (see \cite{kruparev})
\begin{equation} \label{cond_stab}
|\Pi\sigma_k^-|>\Pi\sigma_k^+
\end{equation}

The example \ref{ex-basic} provides a simple case of an asymptotically stable edge cycle, see Fig. \ref{hetcyc36}. We show below that all asymptotically stable edge cycles have the same simple structure.

%

\begin{Theorem}\label{thm:existence edge cycle}
Let hypothesis (H) hold. An edge cycle exists for (\ref{eq-consecutive}) if and only if condition \eqref{eq:cond_signes} holds as well as the following: 
$$\lambda(c_0+c_1(a_1+\dots +a_n) < f_q(1) < \lambda(c_0+c_1(-a_1+\dots - a_{n-1} +a_n).$$
This edge cycle  connects the sequence of $2n$ equilibria 
\begin{equation}\label{eq:edgecycle}
(1,1,\dots,1)\rightarrow(1,1,\dots,-1)\rightarrow\dots (-1,-1,\dots,-1)\rightarrow\dots\rightarrow(-1,1,\dots,1).
\end{equation} 
\end{Theorem}
\begin{proof}
Let $\xi=(x_1,\dots,x_n)$ be an equilibrium in the cycle. Note first that according to (\ref{eigenvalues}), in order to have one unique positive eigenvalue $\sigma_k$ with $k<n$, the following must be true: \eqref{eq:cond_signes} holds and (i) all $x_j$ with $j\leq k$ have the same sign, (ii) $x_kx_{k+1}=-1$ and (iii)  $x_j$ has the sign of $x_{k+1}$ for $k+1<j$. Let $\xi'$ be the next equilibrium in the cycle, then we must have $x'_j=x_j$ for all $j\neq k$ and $x'_k=-x_k$. Observe that we then have $\sigma'_k<0$. It is straightforward to check that under \eqref{eq:cond_signes}, there is no equilibrium point lying on the edge joining $\xi$ to $\xi'$ and therefore that a heteroclinic connection exists on this edge. \\
Now let's assume that the positive eigenvalue is $\sigma_n$. Then all $x_j$'s, $j=1,\dots n$, must be equal and the condition $\sigma_n>0$ can be written $f_q(1)-\lambda(c_0+c_1(a_1+\dots +a_n)>0$. Also we request $x'_n=-x_n$ and  $\sigma'_n<0$, which can be written $f_q(1)-\lambda(c_0+c_1(-a_1+\dots -a_{n-1}+a_n)<0$. As for the case $k<n$ we can check that if these inequalities are satisfied a heteroclinic orbit joins $\xi$ to $\xi'$. \\
From the above construction we deduce that the edge cycle must connect the equilibria in the sequence \eqref{eq:edgecycle}. 
\end{proof}

\begin{Corollary} \label{cor:edgecycle}
Edge cycles are in one-to-one correspondance with connectivity matrices \eqref{eq:Js} with $a_1=-1$ and $a_j=0$ for $j>1$. Moreover, under hypothesis (H), they are asymptotically stable iff $\lambda<c_0/f_q(1)$.
\end{Corollary}  
\begin{proof}
It follows from the above theorem that the matrix $\Sigma$ for an edge cycle is defined by 
$$
\Sigma = (\eta,\eta P,\dots,\eta P^{n-1})^T
$$ 
where $\eta$ is the vector $(v,-v)$ with $v=(1,\dots, 1)$ ($n$ times). Note that $\eta P^n=-\eta$. It follows that the rows of $\Sigma P$ are $\eta P,\dots,\eta P^{n-1}$ and $-\eta$. Hence $J$ with $a_1=-1$ and $a_j=0$ is solution of
$$
J \Sigma=\Sigma P.
$$ 
Since the rows of $\Sigma$ are independant the matrix $\Sigma \Sigma^T$ is invertible, hence the solution is unique. \\
It is straightforward to check that \eqref{cond_stab} is true in this case iff $\lambda<c_0/f_q(1)$.
\end{proof}
Note that Example \ref{ex-basic} provides the simplest case of an edge cycle.

\subsection{Heteroclinic non-edge cycles}
We have seen in the previous section that in order for a vertex equilibrium $\xi=(x_1,\dots,x_n)$ of (\ref{eq-consecutive}) to have a unique positive eigenvalue, a necessary condition was that a change of sign in the sequence of coordinates $x_j$ occurs at most once. The sign of $\sigma_n$ is a special case, it depends on the coefficients $a_1,\dots,a_n$. Suppose now that $\xi$ has two positive eigenvalues, along directions $x_k$ and $x_l$. The corresponding two dimensional unstable manifold lies in the face defined by the fixed coordinates $x_j$ when $j\neq k,l$. Assuming $k<l<n$, the four vertices on this face are $\xi$, $\xi'=(x_1,\dots,-x_k,\dots,x_n)$, $\xi''=(x_1,\dots,-x_l,\dots,x_n)$ and $\tilde\xi=(x_1,\dots,-x_k,\dots,-x_l,\dots,x_n)$. The question which we address now is whether there can exist stable heteroclinic cycles which involve saddle-sink connecting trajectories from $\xi$ to $\tilde\xi$. This situation can of course be generalized to more than two unstable eigenvalues, if there are more than two switches of sign in the $x_j$'s. Let us first look at an example in low dimension. \\
In all the following we assume hypothesis (H) holds.
\medskip

\begin{example} \label{ex:non-edge1}
Consider 3 neurons ($n=3$) and 4 equilibria such that  $\eta=(1, 1, -1, -1)$. Defining $P$ as before and $\eta=(1,1,-1,-1)$, we build $\Sigma$ to form a consecutive cycle with $n=3$ and $p=4$: $\{\eta,\, \eta P,\eta P^2\}$. Hence
\begin{equation}\label{eq-seq4}
\Sigma=\left (\begin{array}{rrrr}1&1&-1&-1\\1&-1&-1&1\\-1&-1&1&1\end{array}\right ).
\end{equation}
Clearly the third row is the opposite of the first one, hence this matrix has rank 2. Nevertheless the cycle is admissible because $rank(\Sigma)=rank(\eta)$ where $rank(\eta)$ is the rank of the matrix $[\eta^T,(\eta P)^T,(\eta P^2)^T,(\eta P^3)^T]^T$ (Theorem 2 of \cite{chuanetal}). Since $\eta P^3=-\eta-\eta P-\eta P^2$, it follows that
\begin{equation}\label{eq:J34}
 J=\left (\begin{array}{rrrr} 0&1&0\\0&0&1\\-1&-1&-1\end{array}\right ).
 \end{equation}
 Note that this example illustrates the criterion derived in Section 3.2, Example \ref{ex-xm1}.
 The equations read
 \begin{align}\label{eq-nonedge_n=3p=4}
 \begin{split}
 \dot x_1&=(1-x_1^2)\left(\lambda c_0x_1 + \lambda c_1x_2 - f_q(x_1)\right )\\
 \dot x_2&=(1-x_2^2)\left(\lambda c_0x_2 + \lambda c_1x_3 - f_q(x_2)\right )\\
 \dot x_3 &=(1-x_3^2)\left(\lambda c_0x_3 - \lambda c_1(x_1+x_2+x_3)- f_q(x_1)\right )
 \end{split}
 \end{align}
Numerical simulations exhibit a heteroclinic cycle for \eqref{eq-stocyc} with $J$ given above, see Fig. \ref{fig:non edge_hetcyc_n=3}.
 \begin{figure}[h]
\center\includegraphics[scale=0.7]{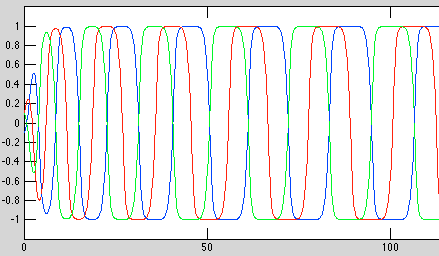}
\caption{Simulations of \eqref{eq-stocyc} with $c_0=0.6$, $\lambda=2.5$, random initial conditions close to origin. Colour code: blue$=x_1$, red$=x_2$, green$=x3$.}
\label{fig:non edge_hetcyc_n=3}
\end{figure} 
Observe that after short transient time $x_1$ and $x_3$ are opposite and move (in opposite directions) while $x_2$ is fixed at $\pm 1$. This indicates that a heteroclinic orbit (if it exists) connects opposite vertices in the faces $x_2=\pm 1$. Now if we set $x_3=-x_1$ in \eqref{eq-nonedge_n=3} with $x_2=\pm 1$, we see that the first and third equations are identical. Therefore the diagonal axis joining the vertices $(1,\pm 1,-1)$ to $(-1,\pm 1,1)$ is flow-invariant. Moreover the eigenvalues at opposite vertices along each of these axes have opposite signs as in the previous sections, showing that a saddle-sink connection exists on these diagonal axes. A more detailed calculation shows that on each of these faces, the dynamics looks like in Fig. \ref{fig:dynamics_face}.  
\begin{figure}[h]
\center\includegraphics[scale=0.3]{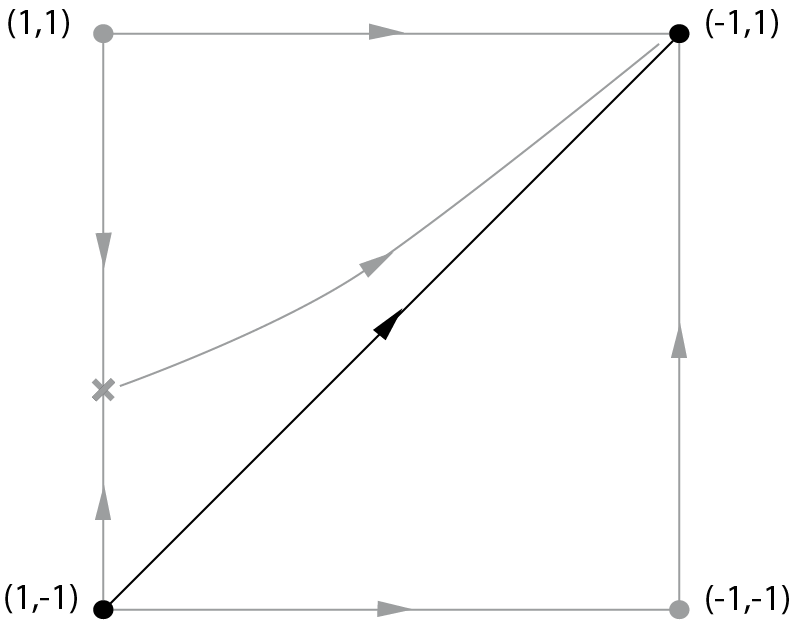}
\caption{Sketch of phase portrait on the face $x_2=-1$ for \eqref{eq-nonedge_n=3}. The cross indicates an equilibrium lying on the edge off vertices.}
\label{fig:dynamics_face}
\end{figure}
Let us explain why the dynamics aligns itself on the diagonal connection (in black in Fig. \ref{fig:dynamics_face}). The unstable eigenvalue at the point $(-1,-1,-1)$ is given by $\sigma_n$ ($n=3$) in \eqref{eigenvalues}: $\sigma_3=2\lambda(-c_0+3c_1)+2f_q(1)$.
\end{example}

We now come back to the general problem.
\begin{lemma}\label{lemma:condition_unstableeigenvalues}
Let $\xi=(x_1,\ldots,x_k,\ldots,x_l,\ldots,x_n)$ be a vertex equilibrium for \eqref{eq-consecutive}. If there are $m$ switches of sign in the sequence $x_1,\ldots,x_n$, then $\xi$ has either $m$ or $m+1$ unstable directions. The latter case occurs if $\sigma_n>0$ in \eqref{eigenvalues}. The unstable manifold of $\xi$ is contained in the hyperface generated by its unstable eigenvectors. 
\end{lemma}
\begin{proof}
One already knows that if $j<n$, then $\sigma_j>0$ iff $x_jx_{j+1}<0$. However when $j=n$, the sign of the eigenvalue depends upon the coefficients $a_1,\ldots,a_n$. The last claim is straightforward from \eqref{eq-consecutive}.
\end{proof}
\begin{lemma}\label{lemma:condition_2Dunstable}
An equilibrium $\xi$ possesses a 2-dimensional unstable manifold if and only if the two conditions are satisfied: (i) either the sequence of coordinates in $\xi$ undergoes two switches of signs and $\sigma_n<0$, or one switch of signs and $\sigma_n>0$; (ii) if $x_k$ and $x_l$ are the unstable directions, then $x_kx_l<0$. 
\end{lemma}
\begin{proof} Condition (i) is clear from Lemma \ref{lemma:condition_unstableeigenvalues}. If condition (ii) is not satisfied, then an additional change of sign must occur somewhere between $x_k$ and $x_l$ and therefore an additional positive eigenvalue must exist.
\end{proof}

The next lemma characterizes when $\sigma_n>0$ when $\xi$ is a column vector of a consecutive cycle $\Sigma$.
\begin{lemma}
Let $\xi=(x_1,\ldots,x_n)$ be an equilibrium in a consecutive cycle. We write $\xi'=J\xi=(x'_1,\ldots,x'_n)^T$ where $J$ is the connectivity matrix \eqref{eq:Js}. Then under the condition \eqref{eq:cond_signes}, $\sigma_n>0$ if and only if $x_nx'_n<0$.
\end{lemma}
\begin{proof}
Recall that $\sigma_n = -2\left(\lambda(c_0 + c_1\sum_{j=1}^n a_jx_j) - f_q(1)\right)$. By construction $\sum_{j=1}^n a_jx_j=x'_n$, therefore by \eqref{eigenvalues}, $\sigma_n = -2\left(\lambda(c_0 + c_1x_nx'_n) - f_q(1)\right)$. The claim follows by \eqref{eq:cond_signes}.
\end{proof}
\begin{lemma}
Let $\xi=(x_1,\ldots,x_n)$ be a vertex equilibrium in a simple consecutive cycle with connectivity matrix $J$ and $\xi'=J\xi=(x'_1,\ldots,x'_n)$. Suppose that $\xi$ possesses a 2-dimensional unstable manifold along directions $x_k$ and $x_l$, which by construction implies $x'_k=-x_k$ and $x'_l=-x_l$. Then a heteroclinic saddle-sink connection $\xi\rightarrow \xi'$ exists in the face of coordinates $(x_k,x_l)$ if and only if $l\neq k\pm 1$. Moreover in this case the diagonal segment joining $\xi$ to $\xi'$ is flow-invariant.   
\end{lemma}
\begin{proof}
The face $F_{kl}=\{(x_1,\ldots,x_{k-1},u,x_{k+1},\ldots,x_{l-1},v,x_{l+1}\ldots,x_n)~,~-1\leq u,v\leq 1\}$ is flow-invariant.
Suppose first that $k+1<l$ (the case $l<k+1$ is of course similar). Then equations in $F$ are
 \begin{align}\label{eq-nonedge_n=3}
 \begin{split}
 \dot u&=(1-u^2)\left(\lambda c_0u + \lambda c_1 x'_k - f_q(u)\right) \\
 \dot v&=(1-v^2)\left(\lambda c_0v + \lambda c_1x'_l - f_q(v)\right)
 \end{split}
 \end{align}
 Set $v=-u$. Then the two above equations are identical because by Lemma \ref{lemma:condition_2Dunstable} we have $x_kx_l=-1$, which also implies $x'_kx'_l=-1$. The saddle-sink connection along this segment follows from the same analysis as in the "edge" case. \\
Suppose now that $l=k+1$. Then in $\xi'$ the coordinates of indices $k$, $k+1$ have opposite signs, which implies that the eigenvalue $\sigma'_k$ of $\xi'$ is positive. Therefore $\xi'$ is a saddle or a source in the face joining $\xi$ to $\xi'$, which proves that no saddle-sink connection $\xi\rightarrow\xi'$ can exist. 
\end{proof}
This lemma can be generalized in a straightforward way to more that two unstable eigenvalues. 
\begin{Definition}
Let $\Sigma=(\xi_1,\ldots,\xi_p)$ be a simple, admissible consecutive cycle. $\Sigma$ has {\em adjacent switches} if in one column (at least), the sign of the entries change two or more times consecutively. 
\end{Definition}
The following theorem summarizes the previous results. 
\begin{Theorem}\label{thm1}
Let hypothesis (H) hold.
For the admissible simple consecutive cycle $\Sigma=(\xi_1,\ldots,\xi_p)$ with Hopfield equations \eqref{eq-consecutive}, the equilibria $\xi_1,\ldots,\xi_p$ are connected by a robust heteroclinic cycle if and only if: (i) condition \eqref{eq:cond_signes} is satisfied; (ii) $\Sigma$ has no adjacent switches. The heteroclinic connections $\xi_i\rightarrow\xi_{i+1}$ lie either along the corresponding edge of the unit cube in phase space, or inside the corresponding face, the dimension of which is equal to the number $q$ of switches of sign of coordinates from $\xi_i$ to $\xi_{i+1}$. In the latter case these connections form a $q$-dimensional manifold, and in this manifold one of them is the diagonal segment joining $\xi_i$ to $\xi_{i+1}$.
\end{Theorem}
The example 1 above illustrates this theorem for a network of three neurons. The first column in \eqref{eq-seq4} contains one switch of sign but the second column contains 2 non adjacent switches. The resulting dynamics close to the heteroclinic cycle is shown in Fig. \ref{fig:non edge_hetcyc_n=3}. \\
The next example also concerns a network with three neurons, however it is a counter-example to existence of a heteroclinic cycle. 

\begin{example} \label{ex:non-edge2}
Let's take $\Sigma$ in the following form:								
\begin{equation}\label{Sigma33}
\Sigma=\left (\begin{array}{rrrr}-1&1&1\\1&1&-1\\1&-1&1\end{array}\right ).
\end{equation}
This matrix defines a simple minimal consecutive cycle: it is $3\times 3$ and invertible. Since $P^3=Id$, it is easy to find that
\begin{equation}\label{eq:J33}
 J=\left (\begin{array}{rrrrr} 0&1&0\\0&0&1\\1&0&0\end{array}\right ).
\end{equation}
Observe that in $\Sigma$ the third row has two adjacent switches of sign. The numerical simulation shows a dynamics converging to the equilibrium $(1,1,1)$ (Fig. \ref{fig:non_hetcyc_n=3}).
\begin{figure}[h]
\center\includegraphics[scale=0.65]{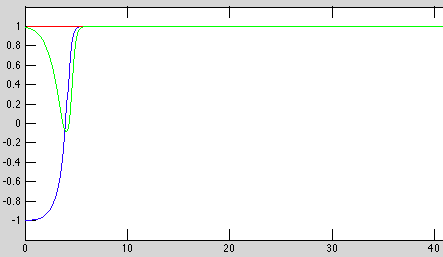}
\caption{Time series for the network with connectivity matrix \eqref{eq:J33}, $c_0=0.6$, $\lambda=2$, initial condition close to the first equilibrium $(-1,1,1)$. Colour code: blue$=x_1$, red$=x_2$, green$=x_3$.}
\label{fig:non_hetcyc_n=3}
\end{figure}
\end{example}
We can't rule out the possibility that condition (ii) in Theorem \ref{thm1} is not satisfied, but the network still possesses a heteroclinic cycle. However in this case the cycle will be different from the one defined by $\Sigma$. Two different $\Sigma$'s can give the same connectivity matrix. Next is an example of this kind.
\medskip

\begin{example} \label{ex:non-edge3}
This example shows that the assumption of Theorem \ref{thm1} concerning the absence of adjacent switches is essential.
Lets consider the following minimal consecutive cycle, which was introduced in \cite{chuanetal} (Example 6).
\begin{equation}\label{tildeSigma56}
\tilde\Sigma=\left (\begin{array}{rrrrrr}1&1&-1&1&-1&-1\\1&-1&1&-1&-1&1\\-1&1&-1&-1&1&1\\1&-1&-1&1&1&-1\\-1&-1&1&1&-1&1\end{array}\right ).
\end{equation}
By Example \ref{ex-xm1} of Section 3.2  the connectivity matrix is
\begin{equation}\label{eq:J56}
J=\left (\begin{array}{rrrrrr}0&1&0&0&0\\0&0&1&0&0\\0&0&0&1&0\\0&0&0&0&1\\-1&-1&-1&-1&-1\end{array}\right ).
\end{equation}
Observe that $\tilde\Sigma$ possesses two adjacent sign switches in the first column. The simulation of the dynamics in this case shows a heteroclinic cycle, however not the one which would correspond to the cycle formed by the columns of $\tilde\Sigma$ (Fig. \eqref{fig:hetcycn=5p=6}), which is consistent with Theorem \ref{thm1}.
\begin{figure}[h]
\center\includegraphics[scale=0.6]{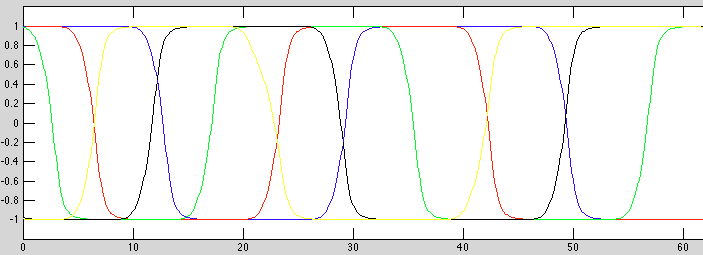}
\caption{Time series (after transients) for the network with connectivity matrix \eqref{eq:J56}, $c_0=0.6$, $\lambda=2.5$, random initial conditions. Colour code: blue$=x_1$, red$=x_2$, green$=x_3$, black$=x_4$, yellow$=x_5$.}
\label{fig:hetcycn=5p=6}
\end{figure}
In the figure we see that $x_3$ moves first from $+1$ to $-1$ while $x_1=x_2=+1$ and $x_4=x_5=-1$, then $x_1$ moves from $+1$ to $-1$ and simultaneously $x_5$ moves from $-1$ to $+1$, then $x_2$ and $x_4$ do the same, and the process repeats itself. The corresponding cycle is given by the following matrix (see also Section \ref{subsec:antisymm cycles}):
\begin{equation}\label{Sigma56}
\Sigma=\left (\begin{array}{rrrrrr}1&1&1&-1&-1&-1\\1&1&-1&-1&-1&1\\1&-1&-1&-1&1&1\\-1&-1&-1&1&1&1\\-1&-1&1&1&1&-1\end{array}\right ).
\end{equation}
\end{example}

In the following example the rank of $\Sigma$ is not maximal. \\
\begin{example} \label{ex:non-edge4}
Let's take
\begin{equation}\label{Sigma46}
\Sigma=\left (\begin{array}{rrrrrr}1&1&-1&-1&-1&1\\1&-1&-1&-1&1&1\\-1&-1&-1&1&1&1\\-1&-1&1&1&1&-1\end{array}\right ).
\end{equation}
Observe that the last row is opposite to the first one, which we call $\eta$, and $rank(\Sigma)=3$. The cycle is admissible and since $\eta P^4=-\eta P$, we have that
\begin{equation}\label{eq:J46}
J=\left (\begin{array}{rrrrrr}0&1&0&0\\0&0&1&0\\0&0&0&1\\0&-1&0&0
\end{array}\right ).
\end{equation}
A numerical simulation is shown in Fig. \ref{fig:hetcyc46}.
\begin{figure}[h]
\center\includegraphics[scale=0.6]{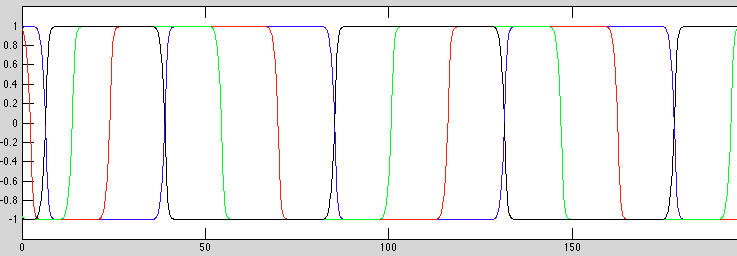}
\caption{Time series (after transients) for the network with connectivity matrix \eqref{eq:J46}, $c_0=0.6$, $\lambda=2.5$, random initial conditions. Colour code: blue$=x_1$, red$=x_2$, green$=x_3$, black$=x_4$.}
\label{fig:hetcyc46}
\end{figure}
\end{example}

\subsection{Some classes of simple consecutive cycles with non-edge heteroclinic cycles}
\subsubsection{Simple consecutive cycles with $p=n$}
Therefore $\Sigma$ is a square matrix. By construction, if $\eta$ denotes the first row of $\Sigma$, then $\Sigma=\left(\eta^T,(\eta P)^T,\ldots,(\eta P^{n-1})^T\right)$. Then it follows from Theorem 2 in \cite{chuanetal} that the cycle is admissible. Moreover $\eta P^n=\eta$, which implies that in \eqref{eq-consecutive}, $a_1=1$ and $a_j=0$ for $j>1$. Hence 
\begin{equation}\label{eq:p=n}
J=\left (\begin{array}{rrrrrr}0&1&0&\ldots&0\\0&0&1&\ldots&0\\\vdots&\ldots&\ldots&\ddots&\vdots\\0&\ldots&\ldots&0&1\\1&0&\ldots&\ldots&0
\end{array}\right ).
\end{equation}
Observe that for $n=3$ adjacent switches always exist in this case (see Example 2).
However for all square simple consecutive cycles of a given dimension $n>3$ and such that no adjacent switches occurs, we can conclude that several non-edge heteroclinic cycles can coexist and their number increases with $n$.
\medskip

\begin{example} \label{ex:non-edge5} 
$n=4$. The only square consecutive cycle with non adjacent switches is generated by $\eta=(1,1,-1,-1)$. Time series of the heteroclinic cycle shown in Fig.\ref{fig:hetcyc44}.
\begin{figure}[h]
\center\includegraphics[scale=0.6]{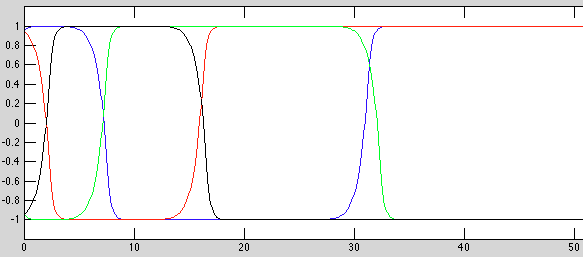}
\caption{Time series (after transients) for the consecutive cycle with $p=n=4$, $c_0=0.6$, $\lambda=3.4$, random initial conditions. Colour code: blue$=x_1$, red$=x_2$, green$=x_3$, black$=x_4$.}
\label{fig:hetcyc44}
\end{figure}
\end{example}

\begin{example} \label{ex:non-edge6} $n=6$. Then the following square consecutive cycles have no adjacent switches: $\eta_1=(1,1,1,-1,-1,-1)$ and $\eta_2=(1,1,1,1,-1,-1)$. The first heteroclinic cycle has connections on three different faces while the second cycle has connections on six different faces.
\end{example}

\subsubsection{Antisymmetric simple consecutive cycles with $p$ even and $n=p-1$} \label{subsec:antisymm cycles}
An {\em antisymmetric cycle}  is generated by a row vector $\eta=(v,-v)$, so that $p$ is even and $v$ has $p/2$ entries. In this case the conditions of Example \ref{ex-xm1} of Section \ref{sec-clahop} hold and 
\begin{equation}\label{eq:antisym}
J=\left (\begin{array}{rrrrrr}0&1&0&\ldots&0\\0&0&1&\ldots&0\\\vdots&\ldots&\ldots&\ddots&\vdots\\0&\ldots&\ldots&0&1\\-1&-1&\ldots&\ldots&-1
\end{array}\right ).
\end{equation}
Moreover by construction $\Sigma$ does not contain adjacent switches. Hence a heteroclinic cycle exists for this matrix.
Two such examples have already been discussed:
with $p=4$ and $n=3$ (Example \ref{ex:non-edge1}), and with $p=6$ and $n=5$ (Example \ref{ex:non-edge3})  (see Fig. \eqref{fig:non edge_hetcyc_n=3} and \eqref{fig:hetcycn=5p=6}, respectively). 
\subsubsection{Simple consecutive cycles with $n<p$ odd given by Propositions \ref{prop-cond1} and \ref{prop-cond2}}
Example \ref{ex-prop2} shows the construction of $J$ for $\eta$ satisfying \eqref{eq-cond1} with $p=15$ and $n=13$. We leave the obvious generalization
of this construction to the reader. Figure \ref{fig:hetcycm151} shows a heteroclinic cycle obtained for $\eta=(1,1,1,1,1,1,1,-1,-1,-1,1,1,1,1,1)$,
$p=15$ and $n=13$. 
\begin{figure}[h]
\includegraphics[scale=0.45]{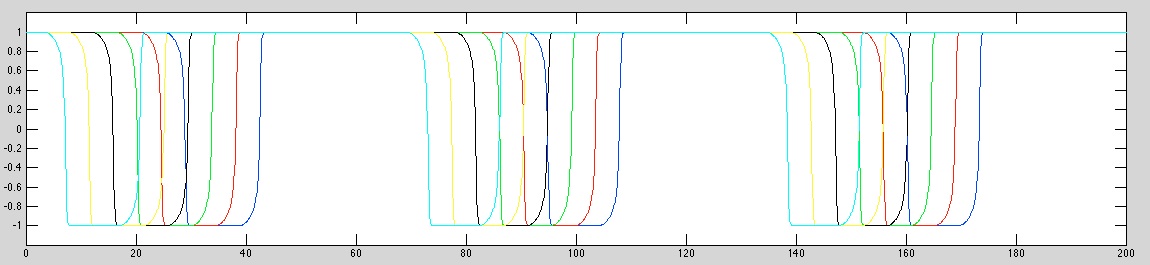}\\\\
\includegraphics[scale=0.45]{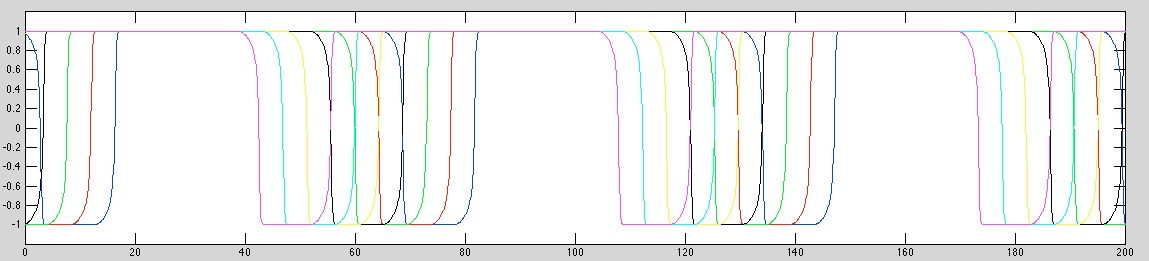}
\caption{Time series (after transients) for the consecutive cycle with $p=15$, $n=13$, $c_0=0.6$, $\lambda=3.4$, random initial conditions. Top picture = time series of $x_1,\dots,x_6$, bottom picture = time series of $x_7,\dots,x_{13}$.
Color code: blue$=x_1,x_7$, red$=x_2,x_8$, green$=x_3,x_9$, black$=x_4,x_{10}$, yellow$=x_5,x_{11}$, cyan$=x_6,x_{12}$, magenta$=x_{13}$.}
\label{fig:hetcycm151}
\end{figure}

For $\eta\in {\rm ker} (P^n-I)$, as stipulated by Proposition \ref{prop-cond2} and Remark \ref{rem-oneblockspace} the heteroclinic cycles
which arise are the same as for $\eta$ truncated to a single block. 
Multiple blocks simply correspond to repeated passages through the same heteroclinic cycle.
For example, if we use, as in Example \ref{ex-prop2} of Section \ref{sec-clahop},
\\$\eta=(1,1,1,-1,-1,1,1,1,-1,-1,1,1,1,-1,-1)$, with $n=5$,
we obtain a heteroclinic cycle in $\R^5$ corresponding to $\eta=(1,1,1,-1,-1)$. A triple passage through this cycle gives $\eta$ as stated above. 

 \section{Conclusion}\label{sec-disc}
 In this work we have studied robust heteroclinic cycles in Hopfield networks with coupling given by
 by the learning rule of \cite{personnazetal}. We gave an extensive classification of heteroclinic cycles
 for couplings of a simple consecutive type. In particular we established a tight relation between the structure of the coupling and the heteroclinic cycles supported by the resulting network. This work is a part of the general program of establishing connections between heteroclinic/homoclinic dynamics and neural processing (see \cite{rabinovichetal} for an outline of this program).
 
 An interesting direction for continuing this work is to determine if a correspondence between cycles in the coupling
 and robust heteroclinic cycles carries over to more realistic settings. As a first attempt of such a generalization we
 intend to introduce delays, as delays arise naturally in neural coupling and can play a functional role \cite{erm-term}.
 Other generalizations include considering systems with noise, more realistic models of neurons, or generalizations
 of the learning rule.\\[1ex]
 \appendix
 \section{Appendix on invariant spaces for linear actions}
 The material presented here is standard in algebra and is related to the rational or Frobenius normal form for matrices, see for example \cite{algebrabook}. 
 We have not found a reference that presents the necessary results in a concise fashion, thus we feel that there is a need for this appendix.
 
 Fix a linear transformation $T\colon\R^n\to\R^n$. We are interested in understanding the ($T$-)invariant subspaces 
of $\R^n$, that is, those subspaces $W$ with $T(W)\subseteq W$, and in particular in identifying the maximal proper
invariant subspaces. Here we will argue that
invariant subspaces are very closely linked to the action of the polynomial ring $\R[x]$ on the linear transformations on
$\R^n$ given by $(f,T)\mapsto f(T)$. Rather than studying specifically the permutation $P$ we consider a more general context
of cyclic transformations. A transformation  $T$ is cyclic if there exists a vector $v\in\R^n$ such that
\[
\R^n={\rm span}\{T^j v, j=0,1,\ldots\}.
\]
We will prove the following result.
\begin{Theorem}\label{th-charinspa}
Let $T\R^n\to\R^n$ be a cyclic transformation. There exists an polynomial $f_T$ of degree $n$ (the minimal polynomial) with the following property.
The invariant spaces for the action of $T$ on $\R^n$ 
are in one to one correspondence with non-trivial factors of the polynomial $f_T$. If $f$ is such a factor then $\ker(f(T))$
is the corresponding invariant space.
\end{Theorem}
It will be clear from the arguments below that $f_P=x^n-1$. Hence Theorem \ref{th-charinvsp} is a direct consequence of Theorem \ref{th-charinspa}.
\subsection*{Generalities on the correspondence between invariant subspaces and polynomials}
For each invariant subspace $W\subseteq\R^n$, we get a map
\[
\R[x]\to L(W,W), f\mapsto f(T)\upharpoonright W
\]
whose kernel is an ideal of $\R[x]$. Since $\R[x]$ is a Principal Ideal Domain (PID) it is a principal ideal. We denote 
the monic generator of this ideal by $f_W$. Note that $W=\ker(f_W(T))$. Note also that for $W=\R^n$, the polynomial $f_{\R^n}$ is the {\it minimal polynomial} 
of $T$, and we will denote it by $f_T$. A few observations in this setting will be 
useful.
\begin{lemma} If $f\in\R[x]$, then both $\ker(f(T))$ and $\Im(f(T))$ are invariant subspaces. 
\end{lemma}
\begin{proof}
This is a consequence of the fact that, for any $h\in\R[x]$, the linear transformations $h(T)$ and $f(T)$ 
commute so that: 
$w\in\ker(f(T))$ implies $f(T)(T(w))=T(f(T)(w))=T(0)=0$ and $w=f(T)(v)$ implies $T(w)=T(f(T)(v))=f(T)(T(v))\in\Im(f(T))$. 
\end{proof}
\begin{lemma} If $f\in\R[x]$ and $f'=\gcd(f,f_T)$, then 
\[
\ker(f(T))=\ker(f'(T))\quad\text{ and }\quad\Im(f(T))=\Im(f'(T)).
\]
\end{lemma}
\begin{proof}
To see this, first we can write $f'=gf+hf_T$ with $g,h\in\R[x]$ by the Euclidean Algorithm. It follows that 
$f'(T)=g(T)\circ f(T)=f(T)\circ g(T)$ since $f_T(T)=0$ by definition of the minimal polynomial. Also, since $f'$ divides $f$, 
there is $k\in\R[x]$ so that $f=kf'$ and thus $f(T)=k(T)\circ f'(T)=f'(T)\circ k(T)$. For the statement about images we have
\begin{align*}
f(T)=f'(T)\circ k(T)\quad\implies\quad \Im(f(T))\subseteq\Im(f'(T))\\
f'(T)=f(T)\circ g(T)\quad\implies\quad \Im(f'(T))\subseteq\Im(f(T)),
\end{align*}
and thus $f(T)$ and $f'(T)$ have equal images. The proof for kernels is similar.
\end{proof}
\begin{Corollary}\label{cor-dir} If $\gcd(f,g)=1$, then $\ker(f(T))\cap\ker(g(T))=\{0\}$.\end{Corollary}
\begin{proof} Suppose $v\in\ker(f(T))\cap\ker(g(T)$. We write $1=hf+kg$ and thus $$v=h(T)(f(T)(v))+k(T)(g(T)(v))=h(T)(0)+k(T)(0)=0$$\end{proof}
\subsection*{Cyclic subspaces}
For $v\in\R^n$ let
\[
[v]_T={\rm span}\;\{v,\, TV,\, T^2v,\, \ldots\}.
\] 
We refer to $[v]_T$ as the cyclic subspace generated by $v$. Clearly, every minimal subspace must be of this form.
We will prove in Lemma \ref{lem-cyc} that every invariant subspace has this form.

Now let $v\in\R^n$ and, again from the action of the polynomial ring, we obtain a map
\[
\R[x]\to\R^n, f\mapsto f(T)(v).
\]
Again, this map is linear, and its kernel is a principal ideal of $\R[x]$, the monic generator of which we will denote by $f_v$.
The image of this map is $[v]_T$.
\begin{lemma}\label{lem-quo} Let $v\in\R^n$, then $f_v=f_{[v]_T}$ and $\dim([v]_T)=\deg(f_v)$.\end{lemma}
\begin{proof}
The first statement follows as if $f(T)$ annihilates $v$ then it also annihilates $g(T)(v)$ for any $g\in\R[x]$. The second 
statement follows since, by the first isomorphism theorem, we have 
\[
[v]_T\cong\R[x]/{\rm Idl}(f_v)
\]
and the dimension of $\R[x]/{\rm Idl}(f_v)$ is equal to $\deg(f_v)$.\end{proof}
\begin{lemma} \label{lem-facto}. Let $v\in\R^n$ and $f,g\in\R[x]$ with $fg=f_v$. Then $f_w=f$ for $w=g(T)(v)$.\end{lemma}
\begin{proof}
Since $f(T)(w)=f_v(v)=0$, it follows that $f_w$ divides $f$. 
We prove that if $h$ is a monic divisor of $f$ satisfying $h(w)=0$ then $f=h$.
Since $h(w)=0$ is equivalent to $hg(T)(v)=0$ it follows that $f_v=fg$ divides $hg$ and thus ${\rm deg}(f)\le {\rm deg}(h)$.
Since $h$ is monic it follows that $f=h$.
\end{proof}
In the proof of the next lemma, we need the fact that $\R[x]$ is a Unique Factorization Domain (UFD), which means that 
each non-zero polynomial $f\in\R[x]$ may be written as $f=af_1^{n_1}\ldots f_k^{n_k}$ where $a$ is a real number and each 
$f_i$ is an irreducible divisor of $f$ which is also prime (that is, for all $h,k\in\R[x]$, if $f_i$ divides $hk$ then $f_i$ divides $h$ 
or $f_i$ divides $k$). 
\begin{lemma}\label{lem-cyc} Let $W$ be an invariant subspace. Then there is $w\in W$ with $f_w=f_W$.\end{lemma}
\begin{proof}
To see this, first note that $f_W={\rm lcm}(f_w\mid w\in W)$. (lcm denotes the least common multiple).
Thus, for any irreducible divisor $f$ of $f_W$ and for $k$ the largest power of $f$ that divides $f_W$, there must be a $w\in W$
so that $f^k$ divides $f_w$. Now taking $g=f_w/f^k$, we see by Lemma \ref{lem-facto} that $f_{w'}=f^k$ where $w'=g(T)(w)$. Doing this for
each irreducible divisor of $f_W$, the sum of the resulting $w'$s is the required element by Corollary \ref{cor-dir}, since $f_{w'}=f^k$ 
implies $w'\in\ker(f^k(T))$.\end{proof}
\subsection*{Cyclic transformations}
Recall that a linear transformation $T\colon\R^n\to \R^n$ is cyclic if there is a $v\in\R^n$ so that $[v]_T=\R^n$. We call $v$ a cyclic generator.
\begin{lemma}  If $T$ is cyclic and $fg=x^n-1$, then $\dim(\Im(g(T)))=deg(f)$.\end{lemma}
\begin{proof} This follows from Lemmas \ref{lem-quo} and \ref{lem-facto}: Take $v\in\R^n$ with $[v]_T=\R^n$. Then $f_v=f_T$ is the minimal polynomial and 
thus $f=f_w$ for $w=g(T)(v)$ and $\dim([w]_T)=\deg(f)$. Finally note that $[w]_T={\rm span}(T^i(g(T)(v))=g(T)(T^i(v))\mid i\in\N)
=g(T)(\R^n)=\Im(g(T))$.\end{proof}
\begin{lemma}  If $T$ is cyclic and $fg=f_T$, then $\Im(g(T)))=\ker(f(T))$.\end{lemma}
\begin{proof}
For any $w$ with $w=g(T)(w')$ we have $f(T)(w)=f(T)(g(T)(w'))=0$, so $\Im(g(T))\subseteq\ker(f(T))$ and we have
\begin{align*}
n=\deg(f)+\deg(g)&=\dim(\Im(g(T)))+\dim(\Im(f(T)))\\
                            &\leq\dim(\ker(f(T)))+\dim(\Im(f(T)))=n.\\
\end{align*}
Thus $\Im(g(T)))=\ker(f(T))$.\end{proof}

The following result now follows.

\begin{Theorem}\label{th-charfinal}
If $T$ is a cyclic linear operator on $\R^n$ and $v\in\R^n$ is a cyclic generator, then the invariant subspaces
of $\R^n$ for $T$ are in one-to-one correspondence with the pairs $(f,g)$ such that $fg=f_T$. The space corresponding 
to such a pair $(f,g)$ is
\[
[g(T)(v)]=Im(g(T))=\ker(f(T)).
\]
In particular, the minimal invariant subspaces correspond to the pairs $(f,g)$ where $f$ is an irreducible factor of $f_T$
in $\R[x]$ and the maximal invariant subspaces correspond to the pairs $(f,g)$ where $g$ is an irreducible factor of $f_T$
in $\R[x]$ .
\end{Theorem}
Note that Theorem \ref{th-charfinal} implies Theorem \ref{th-charinspa}.

\section*{Acknowledgement}
This work was partially supported by
the European Union Seventh Framework Programme (FP7/2007-2013) under grant
agreement no. 269921 (BrainScaleS), no. 318723 (Mathemacs), and by the ERC
advanced grant NerVi no. 227747.


\newpage
 \begin{thebibliography}{99}
%
\bibitem{ashwinetal} P. Ashwin, O. Karabacak and T. Nowotny , Criteria for robustness of heteroclinic cycles in neural microcircuits, {\em J. Math. Neurosci}. {\bf 1}:13 (2011)
\bibitem{rabinovichPRL} C.  Bick C, M. I. Rabinovich. Dynamical origin of the effective storage capacity in the brain's working memory. {\em Phys Rev Lett}. {\bf 103}(21):  218101, 2009
\bibitem{chossatlauterbach} P. Chossat, R. Lauterbach. {\em Methods in Equivariant Bifurcation and Dynamical Systems}, Advanced Series in Nonlinear Dynamics {\bf 15}, World Scientific, Singapour (2000)
\bibitem{chuanetal} Chuan Zhang, G. Dangelmayr, I. Oprea. Storing cycles in Hopfield-type networks with pseudo inverse learning rule:
Admissibility and network topology. {\em Neural Networks} {\bf 46}, 283-298 (2013).
\bibitem{chuansiads} Chuan Zhang, G. Dangelmayr, I. Oprea. Storing cycles in Hopfield-type networks with pseudoinverse learning rule: retrievability and bifurcation analysis.
Submitted (2013)
\bibitem{algebrabook} David S. Dummit and Richard M. Foote. {\em Abstract Algebra} 3rd Edition, John Wiley \& Sons (2003).
\bibitem{erm-term} B. G. Ermentrout, D. H. Terman. Mathematical Foundations of Neuroscience. {\em Interdisciplinary Applied Mathematics}, Vol. 35, Springer (2010).
\bibitem{fukaitanaka} T. Fukai, S. Tanaka. A Simple Neural Network Exhibiting Selective Activation of Neuronal Ensembles: From Winner-Take-All to Winners-Share-All. {\em Neural Comput.} {\bf 9}: 77-97 (1997).
\bibitem{gencicetal} T. Gencic, M. Lappe, G. Dangelmayr and W. Guettinger. Storing cycles in analog neural networks. {Parallel processing in neural systems and computers}, R. Eckmiller, G. Hartmann \& G. Hause (Eds), 445-450, North Holland (1990).
\bibitem{hofbauersigmund} J. Hofbauer, K. Sigmund. {\em Evolutionary Games and Population Dynamics} , Cambridge University Press (1998).
\bibitem{hopfield}
J.\,J.\,Hopfield, 
Neural networks and physical systems with emergent collective computational abilities,
{\em Proc. Natl. Acad. Sci. USA} {\bf 79}(8): 2554--2558, 1982.
\bibitem{kruparev} M. Krupa. Robust heteroclinic cycles. {\em J. of Nonl. Sci.} {\bf 7}, 129--176 (1997).
\bibitem{personnazetal} L. Personnaz, I. Guyon \& G. Dreyfus. Collective computational properties of neural networks: new learning mechanisms. {\em Physical Review A}, {\bf 34}(5) 4217-4228 (1986).
\bibitem{rabinovichetal} M. P. Rabinovich, P. Varona, A. I. Selverston, H. D. I. Abarbanel. Dynamical Principles in Neuroscience. {\em Reviews of Modern Physics} {\bf 78}(4): 1213-1265 (2006).
\bibitem{rabinovichneuron} A. Szucs, R. Huerta, M. I. Rabinovich, A. I. Selverston. Robust Microcircuit Synchronization by Inhibitory Connections. {\em Neuron}, {\bf 61}: 439-453 (2009).
%
\end{thebibliography}
 \end{document}